\documentclass[11pt,a4paper]{article}
\usepackage{amsmath}
\usepackage{amsthm}
\usepackage{amssymb}
\usepackage{amscd}
\usepackage{graphicx}
\usepackage{epsfig}
\usepackage[matrix,arrow,curve]{xy}

\usepackage{latexsym}
\usepackage{amsfonts}
\input xy
\usepackage{tikz}
\usetikzlibrary{arrows,chains,matrix,positioning,scopes,snakes}
\makeatletter
\tikzset{join/.code=\tikzset{after node path={%
\ifx\tikzchainprevious\pgfutil@empty\else(\tikzchainprevious)%
edge[every join]#1(\tikzchaincurrent)\fi}}}
\makeatother

\tikzset{>=stealth',every on chain/.append style={join},
         every join/.style={->}}
\tikzset{
    >=stealth',
    punkt/.style={
           rectangle,
           rounded corners,
           draw=black, very thick,
           text width=6.5em,
           minimum height=2em,
           text centered},
    pil/.style={
           ->,
           thick,
           shorten <=2pt,
           shorten >=2pt,}
}

\xyoption{all} \tolerance=500

\newtheorem{thm}{Theorem}[section]

\newtheorem{lem}[thm]{Lemma}

\theoremstyle{definition}

\newtheorem{remark}[thm]{Remark}
\newtheorem{definition}[thm]{Definition}

\font\black=cmbx10 \font\sblack=cmbx7 \font\ssblack=cmbx5 \font\blackital=cmmib10  \skewchar\blackital='177
\font\sblackital=cmmib7 \skewchar\sblackital='177 \font\ssblackital=cmmib5 \skewchar\ssblackital='177
\font\sanss=cmss12 \font\ssanss=cmss8 scaled 900 \font\sssanss=cmss8 scaled 600 \font\blackboard=msbm10
\font\sblackboard=msbm7 \font\ssblackboard=msbm5 \font\caligr=eusm10 \font\scaligr=eusm7 \font\sscaligr=eusm5

\font\bsymb=cmsy10 scaled\magstep2
\def\all#1{\setbox0=\hbox{\lower1.5pt\hbox{\bsymb
       \char"38}}\setbox1=\hbox{$_{#1}$} \box0\lower2pt\box1\;}
\def\exi#1{\setbox0=\hbox{\lower1.5pt\hbox{\bsymb \char"39}}
       \setbox1=\hbox{$_{#1}$} \box0\lower2pt\box1\;}

\newfam\bifam
\textfont\bifam=\blackital \scriptfont\bifam=\sblackital \scriptscriptfont\bifam=\ssblackital

\newfam\blfam
\textfont\blfam=\black \scriptfont\blfam=\sblack \scriptscriptfont\blfam=\ssblack

\newfam\bbfam
\textfont\bbfam=\blackboard \scriptfont\bbfam=\sblackboard \scriptscriptfont\bbfam=\ssblackboard

\newfam\ssfam
\textfont\ssfam=\sanss \scriptfont\ssfam=\ssanss \scriptscriptfont\ssfam=\sssanss
\def\sss#1{{\fam\ssfam\relax#1}}

\newfam\clfam
\textfont\clfam=\caligr \scriptfont\clfam=\scaligr \scriptscriptfont\clfam=\sscaligr

\def\pmb#1{\setbox0\hbox{${#1}$} \copy0 \kern-\wd0 \kern.2pt \box0}
\def\pmbb#1{\setbox0\hbox{${#1}$} \copy0 \kern-\wd0
      \kern.2pt \copy0 \kern-\wd0 \kern.2pt \box0}
\def\pmbbb#1{\setbox0\hbox{${#1}$} \copy0 \kern-\wd0
      \kern.2pt \copy0 \kern-\wd0 \kern.2pt
    \copy0 \kern-\wd0 \kern.2pt \box0}
\def\pmxb#1{\setbox0\hbox{${#1}$} \copy0 \kern-\wd0
      \kern.2pt \copy0 \kern-\wd0 \kern.2pt
      \copy0 \kern-\wd0 \kern.2pt \copy0 \kern-\wd0 \kern.2pt \box0}
\def\pmxbb#1{\setbox0\hbox{${#1}$} \copy0 \kern-\wd0 \kern.2pt
      \copy0 \kern-\wd0 \kern.2pt
      \copy0 \kern-\wd0 \kern.2pt \copy0 \kern-\wd0 \kern.2pt
      \copy0 \kern-\wd0 \kern.2pt \box0}

\mathchardef\za="710B  
\mathchardef\zb="710C  
\mathchardef\zg="710D  
\mathchardef\zd="710E  
\mathchardef\zve="710F 
\mathchardef\zz="7110  
\mathchardef\zh="7111  
\mathchardef\zvy="7112 
\mathchardef\zi="7113  
\mathchardef\zk="7114  
\mathchardef\zl="7115  
\mathchardef\zm="7116  
\mathchardef\zn="7117  
\mathchardef\zx="7118  
\mathchardef\zp="7119  
\mathchardef\zr="711A  
\mathchardef\zs="711B  
\mathchardef\zt="711C  
\mathchardef\zu="711D  
\mathchardef\zvf="711E 
\mathchardef\zq="711F  
\mathchardef\zc="7120  
\mathchardef\zw="7121  
\mathchardef\ze="7122  
\mathchardef\zy="7123  
\mathchardef\zf="7124  
\mathchardef\zvr="7125 
\mathchardef\zvs="7126 
\mathchardef\zf="7127  
\mathchardef\zG="7000  
\mathchardef\zD="7001  
\mathchardef\zY="7002  
\mathchardef\zL="7003  
\mathchardef\zX="7004  
\mathchardef\zP="7005  
\mathchardef\zS="7006  
\mathchardef\zU="7007  
\mathchardef\zF="7008  
\mathchardef\zW="700A  

\newcommand{\be}{\begin{equation}}
\newcommand{\ee}{\end{equation}}

\newcommand{\bea}{\begin{eqnarray}}
\newcommand{\eea}{\end{eqnarray}}
\newcommand{\beas}{\begin{eqnarray*}}
\newcommand{\eeas}{\end{eqnarray*}}
\def\*{{\textstyle *}}

\newcommand{\w}{\wedge}

\def\cA{{\cal A}}

\def\cE{{\cal E}}

\def\cJ{{\cal J}}

\def\cO{{\cal O}}
\def\cU{{\cal U}}

\def\sT{{\sss T}}

\def\cM{{\cal M}}

\newdir{|>}{%
!/4.5pt/@{|}*:(1,-.2)@^{>}*:(1,+.2)@_{>}}

\newcommand{\p}{\partial}
\newcommand{\la}{\langle}
\newcommand{\ra}{\rangle}

\newcommand{\Ci}{C^{\infty}}

\newcommand{\N}{\mathbb{N}}
\newcommand{\Z}{\mathbb{Z}}
\newcommand{\R}{\mathbb{R}}

\newcommand{\lp}{\left(}
\newcommand{\rp}{\right)}

\newcommand{\op}[1]{\!\!\mathop{\rm ~#1}\nolimits}

\newcommand{\id}{\op{id}}
\newcommand{\ao}{\mathbf{1}^n}

\begin{document}
\title{\bf Splitting theorem for $\Z_2^n$-supermanifolds}
\date{}
\author{Tiffany Covolo\footnote{National Research University High School of Economics, Moscow, Russia, covolotiffany@gmail.com}, Janusz Grabowski\footnote{Institute of Mathematics, Polish Academy of Sciences, Warsaw, Poland, jagrab@impan.pl}, and Norbert Poncin\footnote{University of Luxembourg, Luxembourg City, Grand-Duchy of Luxembourg, norbert.poncin@uni.lu}}

\maketitle

\begin{abstract}Smooth $\Z_2^n$-supermanifolds have been introduced and studied recently. The corresponding sign rule is given by the `scalar product' of the involved $\Z_2^n$-degrees. It exhibits interesting changes in comparison with the sign rule using the parity of the total degree. With the new rule, nonzero degree even coordinates are not nilpotent, and even (resp., odd) coordinates do not necessarily commute (resp., anticommute) pairwise.
The classical Batchelor-Gaw\c edzki theorem says that any smooth supermanifold is diffeomorphic to the `superization' $\zP E$ of a vector bundle $E$. It is also known that this result fails in the complex analytic category. Hence, it is natural to ask whether an analogous statement goes through in the category of $\Z_2^n$-supermanifolds with its local model made of formal power series. We give a positive answer to this question.
\end{abstract}

\vspace{2mm} \noindent {\bf MSC 2010}: 17A70, 58A50, 13F25, 16L30 \medskip



\section{Introduction}

A few papers on $\Z_2^n$-Superalgebra and $\Z_2^n$-Supergeometry appeared recently \cite{COP12,CGP14,CGP16,Pon16}.\medskip

In standard Supergeometry or $\Z_2$-Supergeometry (resp., in $\Z_2^2$-Supergeometry; $\Z_2^3$-Super\-geometry), one considers coordinates of degrees $0$ and $1$ ($\,$resp., $$(0,0), (1,1), (0,1),\;\text{and}\;(1,0)\;;$$ $$(0,0,0), (0,1,1),(1,0,1), (1,1,0), (0,0,1), (0,1,0), (1,0,0),\;\text{and}\;(1,1,1)\;)\;.$$ In the $\Z_2^n$-case, there exist $2^n$ different degrees, the first $2^{n-1}$ being even and the second $2^{n-1}$ odd. However, the commutation rule for the coordinates is not, as usual, given by the product of the parities, but by the scalar product of the involved degrees. More precisely, if $y$ (resp., $\zh$) is of degree $(0,1,1)$ (resp., $(0,1,0)$), we set \be\label{CommRule}y\cdot\zh=(-1)^{\la (0,1,1),(0,1,0)\ra}\zh\cdot y=-\zh\cdot y\;,\ee where $\la-,-\ra$ denotes the standard scalar product in $\R^3\,.$ This `scalar product commutation rule' implies significant differences with the classical theory: even coordinates may anticommute ($(-1)^{\la (1,1,0),(1,0,1) \ra}=-1$), odd coordinates may commute ($(-1)^{\la (1,0,0),(0,1,0) \ra}=+1$), and nonzero degree even coordinates are not nilpotent ($(-1)^{\la (1,1,0),(1,1,0) \ra}=+1$).\medskip

The study of $\Z_2^n$-gradings, $n\ge 0$, together with the commutation rule (\ref{CommRule}), is in some sense necessary and sufficient. Sufficient, since any sign rule, for any finite number $m$ of coordinates, is of the form (\ref{CommRule}), for some $n\le 2m$ \cite{CGP14}, \cite{CGP16}; and necessary, in view of the needs of Physics, Algebra, and Geometry. In Physics, $\Z_2^n$-gradings, $n\ge 2,$ are used in string theory and in parastatistical supersymmetry \cite{AFT10}, \cite{YJ01}. In Mathematics, there exist good examples of $\Z_2^n$-graded $\Z_2^n$-commutative algebras: the algebra of Deligne differential superforms is $\Z_2^2$-commutative, $$\za \wedge \zb = (-1)^{\deg({\za})\deg({\zb})+\op{p}(\za)\op{p}(\zb)}\zb\wedge\za\;, $$ where $\deg$ (resp., $\op{p}$) denotes the cohomological degree (resp., the parity) of the superforms $\za$ and $\zb$, the algebra $\mathbb{H}$ of quaternions is $\Z_2^3$-commutative, and, more generally, any Clifford algebra $\op{Cl}_{p,q}(\R)$ is $\Z_2^{p+q+1}$-commutative \cite{COP12}, ... And there exist interesting examples of $\Z^n_2$-supermanifolds: the tangent and cotangent bundles $T \cM$ and $T^\star \cM$ of a standard supermanifold, the superization of double vector bundles such as, e.g., $TT M$ and $T^*TM$, where $M$ is a classical purely even manifold, and, more generally, the superization of $n$-vector bundles, ...\medskip

For instance, if $(x,\xi)$ are the coordinates of $\cM$, the coordinates of $T\cM$ are $(x,\xi,\op{d}x,\op{d}\xi)$. As concerns degrees, we have two possibilities. Either, we add the cohomological degree $1$ of $\op{d}$ and the parities $0$ (resp., $1$) of $x$ (resp., $\xi$), or, we keep them separated (richer information). In the first case, the coordinates $(x,\xi,\op{d}x,\op{d}\xi)$ have the parities $(0,1,1,0)$, we use the standard supercommutation rule and obtain a classical supermanifold; in the second, the coordinates $(x,\xi,\op{d}x,\op{d}\xi)$ have the degrees $((0,0),(0,1),(1,0),(1,1))$, we apply the $\Z_2^2$-commutation rule (\ref{CommRule}) and get a $\Z_2^2$-manifold. The local model of the supermanifold $T\cM$ is of course made of the polynomials $\Ci(x,\op{d}\xi)[\xi,\op{dx}]$ in the odd indeterminates with coefficients that are smooth with respect to the even variables. On the other hand, the base of the $\Z_2^2$-manifold $T\cM$ is -- exactly as in $\Z$-graded geometry -- made only of the degree $(0,0)$ variables, whereas, with respect to the other indeterminates, we consider not only polynomials, but all power series $\Ci(x)[[\xi,\op{d}x,\op{d}\xi]]$.\medskip

Let us comment on the latter local model. Consider an arbitrary $\Z_2^2$-manifold with coordinates $(x,y,\xi,\zh)$ of degrees $((0,0),(1,1),(0,1),(1,0))$, and let $$\phi: \, \{x,y,\xi,\zh\}\mapsto \{x',y',\xi',\zh'\}$$ be the coordinate transformation \be x'= x+y^2,\quad y'=y,\quad \xi'=\xi,\quad \zh'=\zh\;.\label{CoordTransEx}\ee Note that (\ref{CoordTransEx}) respects the $\Z_2^2$-degree and that $y$ is, as mentioned above, {\it not} nilpotent. If we now change coordinates in a target function of the type $F(x')$, we get, using as usual a formal Taylor expansion,
$$F(x')= F(x+y^2) = \sum_{\za}\frac{1}{\za !} (\p_{x'}^{\za} F)(x) y^{2\za}\;,$$ where the {\small RHS} is really a series, precisely because $y$ is not nilpotent. However, the pullback of a target function must be a source function. The only way out is to decide that functions are formal series, thus opting for the aforechosen local model \be\label{LocMod}(U,\,\Ci(x)[[y,\xi,\zh]])\;,\ee where $U$ is open in some $\R^p$ (of course, since $\xi$ and $\zh$ {\it are} nilpotent, they appear in the series with exponent $0$ or $1$). With this in mind, one easily sees that the most general coordinate transformation is
\be\label{CoordTransGen}
\begin{cases}
x'= {\sum_r f^{x'}_r(x)y^{2r}+\sum_r g^{x'}_r(x)y^{2r+1}\xi\zh}\\
y'= \sum_r f^{y'}_r(x)y^{2r+1}+\sum_r g^{y'}_r(x) y^{2r}\xi\zh\\
\xi'=\sum_r f^{\xi'}_r(x)y^{2r}\xi+\sum_r g^{\xi'}_r(x)y^{2r+1}\zh\\
\zh'= \sum_r f^{\zh'}_r(x)y^{2r}\zh+\sum_r g^{\zh'}_r(x)y^{2r+1}\xi\;\;\;\;, \\
\end{cases}
\ee where $r\in\N$, so that all sums are series, and where the coefficients are smooth in $x$. Let us stress that if we perform this general coordinate transformation (with series) in a target function (that is itself a series), we might a priori obtain series of smooth coefficients, which would then lead to convergence conditions. Fortunately, one can show that this problem does not appear.\medskip

{Note also that the Jacobian matrix that corresponds to (\ref{CoordTransGen}) is of the type
\be\label{Jacobian}
\begin{tikzpicture}
\node[anchor=east] at (-0.5,2) {$\p_{(x,y,\xi,\zh)}(x',y',\xi',\zh') =$};
{\draw (0,0.4) to[out=110,in=-110] (0,3.6); }
{\draw (4,0.4) to[out=70,in=-70] (4,3.6); }
{\draw[very thick] (2,0.4) -- (2,3.6); }
{\draw[very thick] (0,2) -- (4,2); }
	{\draw (0,1.2) -- (4,1.2); }
	{\draw (0,2) -- (4,2); }
	{\draw (0,2.8) -- (4,2.8); }
	{\draw (1,0.4) -- (1,3.6); }
	{\draw (2,0.4) -- (2,3.6); }
	{\draw (3,0.4) -- (3,3.6); }
	\node at (0.5,3.2) {${(0,0)}$};
	\node at (1.5,3.2) {${(1,1)}$};
	\node at (2.5,3.2) {${(0,1)}$};
	\node at (3.5,3.2) {${(1,0)}$};
	\node at (0.5,2.4) {${(1,1)}$};
	\node at (1.5,2.4) {${(0,0)}$};
	\node at (2.5,2.4) {${(1,0)}$};
	\node at (3.5,2.4) {${(0,1)}$};
	\node at (0.5,1.6) {${(1,0)}$};
	\node at (1.5,1.6) {${(0,1)}$};
	\node at (2.5,1.6) {${(0,0)}$};
	\node at (3.5,1.6) {${(1,1)}$};
	\node at (0.5,0.8) {${(0,1)}$};
	\node at (1.5,0.8) {${(1,0)}$};
	\node at (2.5,0.8) {${(1,1)}$};
	\node at (3.5,0.8) {${(0,0)}$};
\end{tikzpicture}\;\;,
\ee
i.e., is a block matrix, where all the entries of a same block have the same $\Z_2^2$-degree. Since the classical determinant only works for matrices with commuting entries, and the entries of a Jacobian matrix as above do not necessarily commute, we have to look for an appropriate determinant, i.e., for the $\Z_2^2$-, or, more generally, the $\Z_2^n$-Berezinian. This higher Berezinian has been constructed in \cite{COP12}. For the corresponding new integration theory, we refer to \cite{Pon16} and \cite{GKP16}.\medskip }

{Assume now that the elements of $\Z_2^n$ are ordered lexicographically. Let $n,p,q_1,\ldots,q_{2^n-1}\in\N$ and set $\mathbf{q}=(q_1,\ldots,q_{2^n-1})$. Consider $p$ coordinates $x^1,\ldots,x^p$ of degree $s_0=0\in\Z_2^n$ (resp., $q_1$ coordinates $\xi^1,\ldots,\xi^{q_1}$ of degree $s_1\in\Z_2^n$, $q_2$ coordinates $\xi^{q_1+1},\ldots,\xi^{q_1+q_2}$ of degree $s_2$, ...) and denote by $x=(x^1,\ldots,x^p)$ (resp., $\xi=(\xi^1,\ldots,\xi^q)$) the tuple of all the zero degree (resp., all the nonzero degree) coordinates (of course $q=\sum_kq_k$). These coordinates $u=(x,\xi)$ commute according to the already mentioned rule (\ref{CommRule}). More precisely, if $u^\za$ and $u^\zb$ are coordinates of $\Z_2^n$-degree $s_k$ and $s_\ell$, respectively, we have the $\Z_2^n$-commutation rule
\be\label{Z2n-com} u^\za u^\zb = (-1)^{\la s_k,s_\ell\ra} u^\zb u^\za\;.
\ee }

A \emph{$\Z_2^n$-superdomain} of dimension $p|\mathbf{q}$ is a \emph{ringed space} ${\cU^{\,p|\mathbf{q}}} = (U,{\cO}_{U})$, where $U\subset\R^p$ is the open range of $x$, and where the structure sheaf is defined over any open $V\subset U$ as the $\Z_2^n$-commutative associative unital $\R$-algebra
\be\label{local}
{\cO}_U(V)=C^\infty_U(V)[[\zx^1,\dots,\zx^q]]
\ee
of formal power series
\be\label{pf1}
 f(x,\zx)= \sum_{|\mu|=0}^{\infty}   f_{\mu_1\ldots\mu_q}(x)\;(\zx^1)^{\mu_1}\ldots(\zx^q)^{\mu_q} =\sum_{|\mu|=0}^{\infty} f_\mu(x) \zx^\mu\;
\ee
in the formal variables $\xi^1,\dots,\xi^q$ with coefficients in $\Ci_U(V)$ (standard multiindex notation).\medskip

We refer to any ringed space of $\Z_2^n$-commutative associative unital $\R$-algebras as a \emph{$\Z_2^n$-ringed space} and to the functions (\ref{pf1}) as the {\em local $\Z_2^n$-superfunctions} (or just $\Z_2^n$-functions).\medskip

A \emph{$\Z_2^n$-supermanifold} of dimension $p|\mathbf{q}$ is a $\Z_2^n$-ringed space $(M,\cA_M)$ locally isomorphic to a $\Z_2^n$-superdomain of dimension $p|\mathbf{q}$.\medskip

For additional motivation, the discussion of Neklyudova's equivalence \cite{Lei11}, as well as for details on $\Z_2^n$-supermanifolds, their morphisms, and the $\Z_2^n$-Berezinian, we refer the reader to \cite{CGP16} and \cite{COP12}.\medskip

The prototypical classical smooth supermanifold is the `superization of a vector bundle', i.e., the locally super ringed space $(M,\zG(\w E^*))$, where $E$ is a vector bundle over $M$. It is usually denoted $\zP E$ or $E[1]$ and viewed as the total space $E$ with fiber coordinates of parity $1$. We refer to a supermanifold $\zP E=E[1]$ induced by a vector bundle as to a \emph{split supermanifold}, since in this case the algebra of superfunctions splits canonically into the subalgebra of smooth functions on $M$ and the ideal of  nilpotent elements. The importance of the example relies on the fact that any supermanifold is of this type: for any smooth supermanifold $\cM=(M,\cA)$ over a smooth classical manifold $M$, there exists a vector bundle $E$ over $M$, such that $\cal M$ is (noncanonically) diffeomorphic to $\zP E$. The bundle $E$ can be interpreted in terms of the normal bundle of the carrier manifold $M$ \cite{Vor}. A variant of this splitting theorem, which is usually attributed to M. Batchelor \cite{Bat1}, \cite{Bat2}, had already been proven a bit earlier by K. Gaw\c edzki \cite{Gaw77} and known to Berezin \cite{Ber79, Ber83,Ber87}. Moreover, D. Leites informed us that also A. A. Kirillov and A. N. Rudakov convinced themselves independently of the correctness of the claim. Meanwhile, many authors wrote about the statement (using often different approaches), e.g., \cite{BR84} and \cite{Man}, to cite at least two.\medskip

A similar proposition holds for $\N$-manifolds \cite{BP12}: Any smooth $\N$-manifold $\cM=(M,\cA)$ of degree $n$, $n\in \N\setminus\{0\}$, is noncanonically diffeomorphic to a split $\N$-manifold $\zP E$, where $E=\bigoplus_{i=1}^nE_{-i}$ is a graded vector bundle over $M$ concentrated in degrees $-1,\ldots,-n$, and where $\zP E=\bigoplus_{i=1}^nE_{-i}[i]$, what means that the fiber coordinates of $E_{-i}$ are viewed as having degree $i\in\N$. With a use of similar methods one can easily prove also a splitting theorem for $n$-fold vector bundles: any $n$-fold vector bundle $E$ is noncanonically isomorphic with the direct sum (over $M$) of vector bundles
\be\label{nv-split}
E\simeq \bigoplus_{i\in\Z_2^n\setminus \{0\}}E_i\,.
\ee
We refer to Section \ref{nonsplit} for more details.\medskip

On the other hand, Batchelor-Gaw\c edzki theorem does not hold for complex analytic supermanifolds: there exist holomorphic supermanifolds whose structure sheaf is NOT isomorphic to the sheaf of sections of a bundle of exterior algebras \cite{Gre82}.\medskip

The goal of this text is to show that an analog of the Batchelor-Gaw\c edzki result holds true for $\Z_2^n$-supermanifolds which, in view of their local model made of formal power series in all nonzero degree coordinates, remind one of the analytic case.\medskip

\noindent{\bf Sources}. Of course, there is an extensive literature on Supergeometry and related topics and it is impossible to give complete references. The sources that had an impact on the present text are: \cite{DAL}, \cite{Lei11}, \cite{Man}, \cite{DM99}, \cite{CCF}, \cite{DSB}, \cite{BP12}, \cite{GKP1}, \cite{GKP2}, and \cite{GKP12}.

\section{Split and nonsplit $\Z_2^n$-supermanifolds}

\subsection{Split $\Z_2^n$-supermanifolds}

Start with a $\Z_2^2\setminus\{0\}$-graded vector bundle $E=E_{01}\oplus E_{10}\oplus E_{11}$ over a manifold $M$, and set $$\zP E=E_{01}[01]\oplus E_{10}[10]\oplus E_{11}[11]\;,$$ where the degrees in the square brackets are assigned to the fiber coordinates (since coordinate transformations are linear, this assignment is of course consistent). Denote by $\odot^k (\zP E)^*$, $k\ge 2,$ the $\Z_2^2$-graded symmetric ($\Z_2^2$-commutative) $k$-tensor bundle of $(\zP E)^*$ and consider the function sheaf \be\label{GradMfdDVB}\cA(\zP E):=\prod_{k\ge 0}\zG(\odot^k(\zP E)^*)\;.\ee The limit $\cA(\zP E)$ is a sheaf of $\Z_2^2$-graded $\Ci$-modules, as well as a sheaf of $\Z_2^2$-superalgebras. The multiplication $\odot$ is the standard one: when writing formal series $\sum_{k=0}^\infty\Psi_k$ instead of families $(\Psi_0,\Psi_1,\ldots)$, we get \be\label{MultFormSer}\sum_{k}\Psi_k'\odot\sum_\ell\Psi_\ell''=
\sum_n\sum_{k+\ell=n}\Psi_k'\odot\Psi_\ell''\;.\ee

\noindent In addition, the sheaf $\cA(\zP E)$ is locally canonically isomorphic to $\Ci_{\R^p}[[\xi,\zh,\zy]]$, where $p=\dim M$ and where $\xi$, $\zh$, and $\zy$ are the fiber coordinates of $E_{01}$, $E_{10}$, and $E_{11}$, respectively. Hence, the pair $(M,\cA(\zP E))$ is a $\Z_2^2$-supermanifold. The assignment (\ref{GradMfdDVB}) can easily be extended to $\Z_2^n\setminus\{0\}$-graded vector bundles and $\Z_2^n$-supermanifolds.

\begin{definition} We refer to a $\Z_2^n$-supermanifold $(M,\cA(\zP E))$, which is implemented by a $\Z_2^n\setminus\{0\}$-graded vector bundle $E$ over $M$, as a \emph{split $\Z_2^n$-supermanifold}.\end{definition}

\subsection{Superizations of $n$-fold vector bundles and nonsplit $\Z_2^n$-supermanifolds}\label{nonsplit}

In \cite{CGP16}, we showed that the $\Z_2^n$-superization of any $n$-fold vector bundle, $n\ge 1$, leads to a $\Z_2^n$-supermanifold.
Recall that an $n$-fold vector bundle is a manifold $E$ equipped with $n$ compatible vector bundle structures. The compatibility condition means that the corresponding Euler vector fields pairwise commute \cite{GR09}. The Euler vector fields induce an $\N^n$-grading in the structure sheaf of $E$ and one can choose an atlas consisting of homogeneous functions of degrees $\le \ao$ with respect to the lexicographical order, where $\ao=(1,\dots, 1)\in\{ 0,1\}^n\simeq \Z_2^n$. We construct the $\Z^n_2$-superization $\zP E$ using the same local coordinates and transformation rules but requiring the sign rules (\ref{Z2n-com}) instead of the commutation. The structure is consistent, because it turns out that the factors in products of coordinates appearing in the transformation rules $\Z_2^n$-commute, so that the cocycle condition remains valid in the $\Z_2^n$-commutative setup.

Now, according to the splitting theorem for $n$-fold vector bundles, we have a noncanonical identification (\ref{nv-split}), which implies a splitting theorem for $\zP E$. However, as mentioned before,
generally $n$-fold vector bundles do not split canonically, so $\zP E$ is generally not canonically split.

In particular, for a vector bundle $V$ over $M$, the tangent bundle $E=\sT V$ is known to be canonically a double vector bundle. It is isomorphic to $V\oplus_M V\oplus_M \sT M$, but there is no canonical identification
$$\sT V\simeq V\oplus_M V\oplus_M \sT M\;$$ of double vector bundles, in general. Let us further emphasize that the canonical vector bundle structure of the {\small RHS} over $M$ is not part of its double vector bundle structure.

\section{Batchelor-Gaw\c edzki theorem}

To our knowledge, even in the case of classical supermanifolds, only a small number of proofs of the Batchelor-Gaw\c edzki theorem, which are neither too short (and therefore difficult to understand), nor too long (and therefore time consuming to read), can be found in the literature. Below, we expand the half-page cohomological proof of \cite{Man} and extend it from the setting of standard supermanifolds to the formal series context of $\Z_2^n$-supermanifolds. As we must deal with non-nilpotent formal variables, the proof needs an additional attention. For information about series in abstract topological algebras (resp., sheaf-theoretic issues), we refer the reader to \cite{CGP14}, Section 7.1 (resp., Sections 7.3 and 7.4, as well as the proof of Proposition 7.7).

{\begin{remark} As a matter of fact the proof of the Batchelor-Gaw\c edzki theorem {\it is} quite involved, as well in the standard super-case, as, a fortiori, in the general $\Z_2^n$-case. In ordinary smooth Supergeometry, this splitting theorem means that the considered supermanifold admits an atlas with coordinate changes $(x,\xi)\rightleftarrows(x',\xi')$ of the type $$x'^i=x'^i(x)\quad\text{and}\quad \xi'^a=\zy^a_b(x)\xi^b\;.$$ In the general situation, this result is equivalent to the statement that any smooth $\Z_2^n$-supermanifold can noncanonically be equipped with an atlas, whose coordinate transformations are of the form $$x'^i=x'^i(x)\quad\text{and}\quad \xi'^a=\zy^a_b(x)\xi^b\;,$$ where $\zy(x)$ is a block diagonal matrix with $q_k\times q_k$ diagonal blocks ($k\in\{1,\ldots,2^n-1\}$). In other words, the $q_k$ coordinates $\xi'^a$ of $\Z_2^n$-degree $s_k$ depend only on the old $q_k$ coordinates $\xi^b$ of the same degree $s_k$. Of course, a direct proof of this result, which does not take advantage of the power of homology and sheaf theories, is highly computational and no attempt to write down such an approach will be made here.\end{remark}}

\subsection{Cohomological invariant}

In the following, we consider sheafs $\cA_M, \Ci_M,\ldots$ over a smooth manifold $M$, but will, for simplicity, just write $\cA,\Ci,\ldots$. Let $\cM=(M,\cA)$ be a $\Z_2^n$-supermanifold, $n\ge 1$, let $\ze:\cA\to \Ci$ be the projection onto $\Ci$, let $\cJ=\ker\ze$, and let $$\cA\supset \cJ\supset \cJ^2\supset\ldots$$ be the decreasing filtration of the structure sheaf by sheaves of $\Z_2^n$-graded ideals.
The quotients $\cJ^{k+1}/\cJ^{k+2}$, $k\ge 0$, are locally finite free sheaves of modules over $\Ci\simeq\cA/\cJ$.
In particular, $${\cal S}:={\cJ}/{\cJ}^2$$ is a locally finite free sheaf of $\Z_2^n\setminus\{0\}$-graded $\Ci$-modules \cite[Example 3.2]{CGP14}.
Hence, there exists a $\Z_2^n\setminus\{0\}$-graded vector bundle $E\to M$ such that $${\cal S}\simeq \zG((\zP E)^*)\;.$$ For instance, in the case $n=2$, we get $${\cal S}\simeq\zG(E_{01}[01]^*\oplus E_{10}[10]^*\oplus E_{11}[11]^*)\;.$$ As above, denote by $\odot$ the $\Z_2^n$-graded symmetric tensor product of $\Z_2^n$-graded $\Ci$-modules and of $\Z_2^n$-graded vector bundles.
Then \be\label{Isos}\zG(\odot^{k+1}(\zP E)^*)\simeq\odot^{k+1}{\cal S}\simeq\cJ^{k+1}/\cJ^{k+2}\;\ee (indeed, the sheaf morphism, which is well-defined on sections by $$\odot^{k+1}\cJ/\cJ^2\ni [s_1]\odot \ldots \odot [s_{k+1}] \mapsto [s_1\cdots s_{k+1}]\in\cJ^{k+1}/\cJ^{k+2}\;,$$ is locally an isomorphism). Our goal is to show that \be\label{Batchelor}\cA(\zP E):=\prod_{k\ge -1}\zG(\odot^{k+1}(\zP E)^*)=\prod_{k\ge -1}\odot^{k+1}{\cal S}\simeq \cA\;\ee as sheaf of $\Z_2^n$-commutative associative unital $\R$-algebras. Indeed, we then have the

\begin{thm} Any smooth $\Z_2^n$-supermanifold is (noncanonically) isomorphic to a split $\Z_2^n$-supermanifold.\end{thm}

It is clear that locally the sheaves (\ref{Batchelor}) coincide. To prove that they are isomorphic, we will build a morphism $\prod_{k\ge -1}\odot^{k+1}{\cal S}\to \cA$ of sheaves of $\Z_2^n$-superalgebras. The idea is to extend a morphism ${\cal S}\to \cA$, or $\cJ/\cJ^2\to \cJ$. The latter will be obtained as a splitting of the sequence $0\to \cJ^2\to \cJ\to \cJ/\cJ^2\to 0$. One of the problems to solve is to show that this sequence can be viewed as a sequence of sheaves of $\Ci$-modules. Therefore, we need an embedding $\Ci\to \cA$.

\subsection{Projection of $\cM$ onto $M$}

We will actually construct a splitting of the short exact sequence $0\to\cJ\to\cA\stackrel{\ze}{\to}\Ci\to 0$, i.e., a morphism $\zf:\Ci\to \cA$ of sheaves of $\Z_2^n$-superalgebras such that $\ze\circ\zf=\op{id}$. More precisely, we build $\zf$ as the limit of an $\N$-indexed sequence of sheaf morphisms $\zf_{k}:\Ci\to \cA/\cJ^{k+1}$:

\[
\begin{tikzpicture}
\matrix(m)[matrix of math nodes, column sep=2em, row sep=4em]
{
&&&\Ci&&&\\
&&&\cA\simeq\varprojlim_k\cA/\cJ^k&&&\\
\ldots&&{\cA}/{\cJ^{k+1}}&&{\cA}/{\cJ^{k+2}}&&\ldots\\
};
\path[->]
(m-3-7)edge node[auto]{$ $} (m-3-5)
(m-3-5)edge node[auto]{$f_{k,k+1} $} (m-3-3)
(m-3-3)edge node[auto]{$ $} (m-3-1)
(m-2-4)edge node[right]{$\pi_k$} (m-3-3)
            edge node[left]{$\pi_{k+1}$} (m-3-5)
(m-1-4)edge[bend right=20] node[left]{$\zf_{k}$} (m-3-3)
            edge[bend left=20] node[right]{$\zf_{k+1}$} (m-3-5)
(m-1-4)edge node[auto]{$\zf$} (m-2-4)
;
\end{tikzpicture}
\]

\noindent In this diagram, the isomorphism $\simeq$ is due to Hausdorff-completeness of abstract $\Z_2^n$-function sheafs and algebras \cite[Section 7.3]{CGP16}. This reference also explains the morphisms $\zp_k$ and $f_{k,k+1}$. It thus suffices to construct the $\zf_k$ so that the upper triangle commutes. This sequence $\zf_k$ will be obtained by induction on $k$, starting from $\zf_0=\id$: we assume that we already got $\zf_{i+1}$ as an extension of $\zf_i$ for $0\le i\le k-1$, and we aim at extending $\zf_k:\Ci\to \cA/\cJ^{k+1}$ to $$\zf_{k+1}:\Ci\to \cA/\cJ^{k+2}\;,$$ in the sense that \be\label{Ext}f_{k,k+1}\circ\zf_{k+1}=\zf_{k}\;.\ee

To that end, we build, for any open subset $\zW\subset M$, an extension $\zf_{k+1,\zW}:\Ci(\zW)\to \cA(\zW)/\cJ^{k+2}(\zW)$ of $\zf_{k,\zW}$, via a consistent construction of extensions of the $\zf_{k,U}$ by local (in the sense of (pre)sheaf morphisms) degree zero unital $\R$-algebra morphisms \be\label{LocExt}\zf_{k+1,U}:\Ci(U)\to \cA(U)/\cJ^{k+2}(U)\simeq\Ci(U)[[\xi^1,\ldots,\xi^q]]_{\le k+1}\ee over a cover $\cal U$ of $\zW$ by $\Z_2^n$-chart domains $U$. Here subscript $\le k+1$ means that we confine ourselves to `series' whose terms contain at most $k+1$ formal parameters. Further, `consistent' means that, if $U,V$ are two domains of the cover, we must have \be\label{Consistency1}\zf_{k+1,U}|_{U\cap V}=\zf_{k+1,V}|_{U\cap V}\;\ee (note that, if we use the identification (\ref{LocExt}), the restricted morphisms (\ref{Consistency1}) are expressed in different coordinate systems).
\begin{lem}
Over any $\Z_2^n$-chart domain $U$, there exists an extension $\zf_{k+1,U}:\Ci(U)\to \cO(U)_{\le k+1}:=\Ci(U)[[\xi^1,\ldots,\xi^q]]_{\le k+1}$ of $\zf_{k,U}$ as local degree zero unital $\R$-algebra morphism.
\end{lem}

\begin{proof} We look for an extension $\zf_{k+1,U}$ of the local degree zero unital $\R$-algebra morphism $\zf_{k,U}:\Ci(U)\to \cO(U)_{\le k}\subset \cO(U)$ (where the latter is built step by step as an extension of $\zf_{0,U}=\op{id}$). Denote the coordinates in $U$ by $x=(x^1,\ldots,x^p)$. In view of \cite[Theorem 7.10]{CGP16}, the `pullbacks' $$\zf_{k,U}(x^i)=x^i+\sum_{1\le|\zm|\le k}f^i_\zm(x)\xi^\zm\in\cO(U)$$ uniquely define a degree zero unital $\R$-algebra morphism $\overline{\zf}_{k,U}:\Ci(U)\to \cO(U)$.
Since the algebra structure in $\cO(U)_{\le k}$ is given by the multiplication of $\cO(U)$ truncated at order $k$, it is easily seen that the restriction $\overline{\zf}_{k,U}|_{\le k}:\Ci(U)\to \cO(U)_{\le k}$ is still a local degree zero unital $\R$-algebra morphism. 
For the same reason, the morphisms $\zf_{k,U}$ and $\overline{\zf}_{k,U}|_{\le k}$ coincide on polynomial functions $P(x)\in\Ci(U)$. We will actually prove that these morphisms coincide on all functions $f(x)\in\Ci(U)$. Then $\overline{\zf}_{k,U}|_{\le k+1}$ is the searched extension $\zf_{k+1,U}$.\medskip

For this a digression is necessary. Let $x_0\in U$ and denote by $$\frak{m}_{x_0}=\{[g]_{x_0}: g(x_0)=0\}\quad\text{and}\quad\frak{m}'_{x_0}=\{[h]_{x_0}: (\ze h)(x_0)=0\}$$ the unique maximal homogeneous ideals of the stalks $\Ci_{x_0}$ and $\cO_{x_0}$ of the sheaves $\Ci$ and $\cO$. The morphism $\overline{\zf}_{k,U}$ (resp., $\overline{\zf}_{k,U}|_{\le k}$, $\zf_{k,U}$) is a local degree zero unital $\R$-algebra morphism (resp., are local degree zero $\R$-linear maps) $\Ci(U)\to\cO(U)$ (all three morphisms are viewed here as valued in $\cO(U)$) and thus defines an algebra morphism (resp., linear maps) $\overline{\zf}_{k,x_0}$ (resp., $\overline{\zf}_{k,x_0}|_{\le k}$, $\zf_{k,x_0}$) between $\Ci_{x_0}$ and $\cO_{x_0}$.\smallskip

{\it The linear maps $\overline{\zf}_{k,x_0}|_{\le k}$ and $\zf_{k,x_0}$ $(\,$induced by the maps we are comparing$\,)$ send $\frak{m}^\ell_{x_0}$ into $\frak{m'}^\ell_{x_0}$, $\ell\ge 1$ $(\,$what is obvious for $\overline{\zf}_{k,x_0}$$\,)$}.\smallskip

Indeed, if $[g]_{x_0}\in\frak{m}_{x_0}^\ell$, then $$[\overline{\zf}_{k,U}(g)]_{x_0}=\overline{\zf}_{k,x_0}[g]_{x_0}\in\frak{m}'^\ell_{x_0}\;,$$ so $$\overline{\zf}_{k,x_0}|_{\le k}[g]_{x_0}=[\overline{\zf}_{k,U}(g)|_{\le k}]_{x_0}\in\frak{m}'^\ell_{x_0}\;,$$ in view of \cite[Lemma 7.6]{CGP16}.

As for $\zf_{k,x_0}$, note first that, if $[g]_{x_0}\in \frak{m}_{x_0}$, then $\ze(\zf_{k,U}g)(x_0)=g(x_0)=0$, so that $\zf_{k,x_0}[g]_{x_0}\in\frak{m}'_{x_0}$. Moreover, if $[g_1]_{x_0},\ldots,[g_\ell]_{x_0}\in\frak{m}_{x_0}$, then $$[\zf_{k,U}(g_1)\ldots\zf_{k,U}(g_\ell)]_{x_0}=[\zf_{k,U}(g_1)]_{x_0}\ldots[\zf_{k,U}(g_\ell)]_{x_0}
=\zf_{k,x_0}[g_1]_{x_0}\ldots\zf_{k,x_0}[g_\ell]_{x_0}\in\frak{m}'^\ell_{x_0}\;,$$ hence $$\zf_{k,x_0}\lp[g_1]_{x_0}\ldots[g_\ell]_{x_0}\rp=\zf_{k,x_0}[g_1\ldots g_\ell]_{x_0}=[\lp\zf_{k,U}(g_1)\ldots\zf_{k,U}(g_\ell)\rp|_{\le k}]_{x_0}\in\frak{m}'^\ell_{x_0}\;.$$

We now come back to the comparison of the morphisms $\zf_{k,U}$ and $\overline{\zf}_{k,U}|_{\le k}\,.$ Consider $f(x)\in\Ci(U)$ and $x_0\in U$, as well as the `series' $$\zf_{k,U}(f)-\overline{\zf}_{k,U}(f)|_{\le k}\in\cO(U)_{\le k}\;.$$ Let $\ell>k$. Theorem \cite[Theorem 7.11]{CGP16} implies that there is a polynomial $P(x)$ such that $[f]_{x_0}-[P]_{x_0}\in\frak{m}_{x_0}^\ell.$ It follows that $$[\zf_{k,U}(f)-\overline{\zf}_{k,U}(f)|_{\le k}]_{x_0}= \zf_{k,x_0}\lp[f]_{x_0}-[P]_{x_0}\rp-\overline{\zf}_{k,x_0}|_{\le k}\lp[f]_{x_0}-[P]_{x_0}\rp\in\frak{m}'^\ell_{x_0}\;.$$ Hence, all the coefficients of $\zf_{k,U}(f)-\overline{\zf}_{k,U}(f)|_{\le k}$ vanish at $x_0$ \cite[Lemma 7.6]{CGP16}, for all $x_0\in U$, and all functions $f(x)\in\Ci(U)$.\end{proof}

To finalize the construction of the sheaf morphism $\zf:\Ci\to \cA$, it now suffices to solve the consistency problem. Let $U$ and $V$ be $\Z_2^n$-chart domains and let $\zf_{k+1,U}$ and $\zf_{k+1,V}$ be the preceding extensions of $\zf_{k,U}$ and $\zf_{k,V}$, respectively.

\smallskip
{\it The difference
\be\label{Consistency3}\zw_{{k+1},UV}(f):=\zf_{k+1,U}|_{U\cap V}(f)-\zf_{k+1,V}|_{U\cap V}(f)\in\cO(U\cap V)_{\le k+1}\;,\ee $f\in\Ci(U\cap V)$,
defines a derivation \be\label{Cochain}\zw_{k+1,UV}:\Ci(U\cap V)\to\cO(U\cap V)_{=k+1}\;.\ee}
\indent At this point, we should remember that the derivation property and the target space of $\zw_{k+1,UV}$ still need an explanation. Further, we should bear in mind that the two terms of the difference in (\ref{Consistency3}) are expressed in different coordinates and that above it is understood that we changed coordinates in the last term.

Concerning the target space, since $\zf_{k+1,U}=\overline{\zf}_{k,U}|_{\le k+1}=\zf_{k,U}+\overline{\zf}_{k,U}|_{= k+1}$, we have $$\zw_{k+1,UV}(f)=\zf_{k,U}|_{U\cap V}(f)+\overline{\zf}_{k,U}|_{=k+1}|_{U\cap V}(f)-\zf_{k,V}|_{U\cap V}(f)-\overline{\zf}_{k,V}|_{=k+1}|_{U\cap V}(f)\;,$$ where after coordinate transformation in the two last terms, we omit all terms of order $> k+1$. Note now that in a coordinate transformation the order cannot decrease, so that the second and fourth terms of the {\small RHS} contain only terms of order $k+1$. The same holds for the difference of the first and third terms. Indeed, since the $\zf_{k,U}$ have already been constructed consistently, they coincide on intersections up to order $k$: the remaining terms are of order $k+1$.

As for the derivation property, start from
\be\label{X}\zf_{k+1,U}|_{U\cap V}(fg)=\zf_{k+1,V}|_{U\cap V}(fg)+\zw_{k+1,UV}(fg)\,\ee and recollect the algebra morphism property of the $\zf_{k+1,U}:\Ci(U)\to \cO(U)_{\le k+1}\,.$
The left hand side equals
\be\begin{split}\label{Y}\zf_{k+1,U}|_{U\cap V}(f)\cdot\zf_{k+1,U}|_{U\cap V}(g)&=\\
\zf_{k+1,V}|_{U\cap V}&(f)\cdot\zf_{k+1,V}|_{U\cap V}(g)+
f\cdot\zw_{k+1,UV}(g)+\zw_{k+1,UV}(f)\cdot g\;.
\end{split}\ee
Indeed, the products are products in $\cO(U\cap V)$ that are truncated at order $k+1$. Comparing (\ref{X}) and (\ref{Y}), we finally get
\be \zw_{k+1,UV}(fg)=\zw_{k+1,UV}(f)\cdot g+f\cdot\zw_{k+1,UV}(g)\;.
\ee
This completes the proof of the claim (\ref{Cochain}).\medskip

In view of (\ref{Isos}), the map $\zw_{k+1,UV}$ is a derivation $$\zw_{k+1,UV}:\Ci(U\cap V)\to (\cJ^{k+1}(U\cap V))^0/(\cJ^{k+2}(U\cap V))^0\simeq \zG(U\cap V,(\odot^{k+1}(\zP E)^*)^{0})\;,$$ i.e., it is a vector field valued in symmetric $(k+1)$-tensors: $$\zw_{k+1,UV}\in \zG(U\cap V,TM\otimes (\odot^{k+1}(\zP E)^*)^{0})\;.$$ This \v{C}ech 1-cochain $\zw_{k+1}$ is obviously a 1-cocycle. However, as well-known, the existence of a partition of unity in $M$ implies that $\check{H}^{\bullet\ge 1}(M,\cE)=0$, for any locally free sheaf $\cE$ over $M$. Hence, there exists a 0-cochain $\zh_{k+1}$, i.e., a family $\zh_{k+1,U}\in \zG(U,TM\otimes(\odot^{k+1}(\zP E)^*)^{0})$, or, still, a family of derivations $$\zh_{k+1,U}:\Ci(U)\to \zG(U,(\odot^{k+1}(\zP E)^*)^{0})\simeq (\cJ^{k+1}(U))^0/(\cJ^{k+2}(U))^0\simeq\cO^0(U)_{= k+1}\;,$$ such that $$\zf_{k+1,U}|_{U\cap V}-\zf_{k+1,V}|_{U\cap V}=\zw_{k+1,UV}=\zh_{k+1,V}|_{U\cap V}-\zh_{k+1,U}|_{U\cap V}\;.$$ It is now easily checked that the sum $\zf'_{k+1,U}:=\zf_{k+1,U}+\zh_{k+1,U}:\Ci(U)\to\cO(U)_{\le k+1}$ is a local degree zero unital $\R$-algebra morphism, which satisfies the consistency condition and extends $\zf_{k,U}$. This proves the existence of the searched morphism $\zf:\Ci\to \cA$ of sheaves of $\Z_2^n$-commutative associative unital $\R$-algebras.\medskip

In fact $\ze\circ\zf=\op{id}$. Indeed, for any open subset $\zW\subset M$ and any $f\in\Ci(\zW)$, we have, on an open cover by $\Z_2^n$-chart domains $U\subset \zW$, $$(\ze_\zW\zf_\zW f)|_U=\ze_U\zf_U(f|_U)=(\op{id}_\zW f)|_U\;.$$ Hence, the

\begin{thm}
For any $\Z_2^n$-supermanifold $(M,\cA_M)$, the short exact sequence $$0\to {\cal J}_M\to \cA_M\stackrel{\ze}{\to }\Ci_M\to 0$$ of sheaves of $\Z_2^n$-commutative associative $\R$-algebras is noncanonically split.
\end{thm}

\subsection{Algebra morphisms}
Due to the embedding $\zf:\Ci\to \cA$, the short exact sequence of sheaves of $\cA$-modules
\be\label{ses}0\to \cJ^2\to \cJ\to {\cal S}=\cJ/\cJ^2\to 0\ee can be viewed as a short exact sequence of sheaves of $\Ci$-modules.
Although $\cJ^2$ and $\cJ$ are {\it not locally finite free}, considering splittings of the short exact sequences $$0\to \cJ^2/\cJ^k\stackrel{i_k}{\longrightarrow} \cJ/\cJ^k\stackrel{p_k}{\longrightarrow} {\cal S}=\cJ/\cJ^2\to 0\;,$$ $k\ge 2$, of locally finite free sheaves of $\Ci$-modules, we can find a splitting $\Phi^1$ of (\ref{ses}).

We now extend $\zF^1$ to a morphism $\zF:\cA(\zP E)=\prod_{k\ge 0}\odot^k{\cal S}\to \cA$ of sheaves of $\Z_2^n$-commutative associative unital $\R$-algebras, putting $\zF:=\zf:\Ci\to \cA$ on $\Ci$, where $\zf$ is the above-constructed degree preserving unital algebra morphism, and
\be\label{Z}\zF(\psi_1\odot\ldots\odot\psi_k):=\zF^1(\psi_1)\cdot\ldots\cdot\zF^1(\psi_k)\in \cJ^k\subset\cA\ee
on $\odot^{k\ge 2}{\cal S}$, with the obvious extension to power series by Hausdorff continuity. This extension is well defined, since the {\small RHS} of (\ref{Z}) is $\Z_2^n$-commutative and $\Ci$-multilinear.

This map $\zF:\cA(\zP E)\to \cA$ respects the degrees and the units, and is an $\R$-algebra morphism, what completes the proof of the theorem.

\section{Acknowledgements}
The research of T. Covolo (resp., J. Grabowski, N. Poncin) was founded by the Luxembourgian NRF grant 2010-1, 786207 (resp., the Polish National Science Centre grant DEC-2012/06/A/ST1/00256, the University of Luxembourg grant GeoAlgPhys 2011-2014). The authors are grateful to S. Morier-Genoud and V. Ovsienko (resp., D. Leites, P. Schapira), whose work was the starting point of this paper (resp., for his suggestions and valuable support, for explanations on sheaf-theoretic aspects).

\end{document}